\documentclass{article}
\usepackage{fixmath}

\usepackage[T1]{fontenc}
\usepackage{ucs}
\usepackage[utf8x]{inputenc}
\usepackage[english]{babel}

\usepackage{microtype}
\usepackage{babel}

\usepackage{authblk}

\usepackage{mathtools} 

\usepackage{enumerate}

\usepackage{geometry} 

\usepackage{amsfonts, amsmath, amsthm, amssymb}
\usepackage{mathrsfs}

\usepackage{physics}

\usepackage{hyperref} 

\usepackage[usenames,dvipsnames]{xcolor}

\newtheorem{theorem}{Theorem}

\newtheorem{proposition}{Proposition}
\newtheorem{corollary}{Corollary}

\theoremstyle{definition}

\theoremstyle{remark}
\newtheorem{remark}{Remark}

\DeclareMathOperator{\Exp}{\mathtt{exponential}}

\DeclareMathOperator{\Gam}{\mathtt{gamma}}

\DeclareMathOperator{\Var}{Var}

\DeclareMathOperator{\Arcsin}{\mathtt{arcsin}}
\DeclareMathOperator{\uniform}{\mathtt{uniform}}

\providecommand{\Exc}[2]{\mathbb{E}\left[#1\,\middle|\, #2\right]}

\providecommand{\card}[1]{\mathtt{\#}#1}

\newcommand{\Ex}{\mathbb{E}}
\providecommand{\Exc}[2]{\mathbb E\left[\right]#1\middle| #2\right]}

\providecommand{\abs}[1]{\lvert#1\rvert}

\newcommand{\conv}{\mathop{\scalebox{1.5}{\raisebox{-0.2ex}{$\ast$}}}} 

\providecommand{\Prob}[1]{\mathbb{P}\left\{#1\right\}}

\providecommand{\keywords}[1]
{
	\small	
	\textbf{\textit{Keywords---}} #1
}

\title{A Poisson representation of the positive sojourn time\\ of L\'evy processes}
\author{Helmut H.~Pitters }
\affil{Mathematics Institute, University of Mannheim}

\begin{document}

\maketitle


\begin{abstract}
   We study the distribution of the positive sojourn time $$ A_t\coloneqq \int_0^t \mathbf 1\{ X_s>0  \}ds $$
of an arbitrary L\'evy process $X\coloneqq (X_t)_{t\geq 0}$. For an exponential random variable $E^{(q)}$ of rate $q>0$ independent of $X$ we show the representation in law
\begin{align*}
  A_{E^{(q)}} =_d \sum_{T\in \Pi} T
\end{align*}
as the sum of points of a Poisson process $\Pi$ with intensity given explicitely in terms of the positivity $t\mapsto \Prob{X_t>0}$. This representation raises some fundamental questions, not least because $\Pi$ turns out to be intimately connected to the celebrated Poisson-Dirichlet distribution. Moreover, we characterise $A_t$ by working out its double Laplace transform, and thus complement a recent result in which the distribution of $A_t$ was characterised via its higher moments. As a Corollary of the Poisson representation, in the special cases where $X$ is Brownian motion, a symmetric stable process, a L\'evy process with constant positivity, we obtain an extension and new derivation of classical (generalised) arcsine laws going back to L\'evy (1939), Kac (1951), and Getoor and Sharpe (1994), respectively. Even in these cases the Poisson representation is new.

 As an application, if $X$ is the $(1/2)$-stable subordinator with drift, we obtain both the Laplace transform of $A_t$ and the density of its distribution. This is the second example of a L\'evy process whose occupation time distribution is known explicitely but is not generalized arcsine, the first example being Brownian motion with drift that was studied earlier in the context of option prizing.
\end{abstract}



\section{Introduction and broader context}\label{sec:introduction}

Let $X\coloneqq (X_t)_{t\geq 0}$ denote a real-valued stochastic process, and consider the time
\begin{align}\label{def:sojourn_time}
  A_t &\coloneqq \int_0^t \mathbf 1\{   X_s>0 \} ds; \qquad t\geq 0,
\end{align}
that $X$ spends above zero up until time $t$. The functional $A_t$ is also referred to as the positive sojourn time of $X,$ or the occupation time of $(0, \infty)$ by $X$.

\emph{Second arcsine law for Brownian motion}. The distribution of
\begin{align}\label{def:arcsine}
  \cos^2(\Theta)\qquad \text{ with }\Theta\sim\uniform(0, 2\pi)
\end{align}
is called the (standard) arcsine distribution. It is supported on the unit interval $(0, 1)$ and has density
\begin{align}\label{eq:arcsine_density}
  x\mapsto \frac{1}{\pi\sqrt{x(1-x)}}\mathbf 1_{(0, 1)}(x).
\end{align}
It bears its name due to its cumulative distribution function $\Prob{A_t/t\leq x}=2\arcsin(\sqrt x)/\pi$ for $0\leq x\leq 1$.

 Brownian motion was the first continuous-time stochastic process for which the distribution of its positive sojourn time~\eqref{def:sojourn_time} was identified, namely by Paul L\'evy~\cite{Levy1939} in 1939 who showed the following.

\begin{theorem}[Levy's second arcsine law for Brownian motion (1939), cf.~\cite{Levy1939}]\label{thm:levy_arcsine}
  Let $X$ denote standard Brownian motion started in zero with positive sojourn time $A_t$ as defined in~\eqref{def:sojourn_time}. Then the proportion $A_t/t$ of time that $X$ spends above zero up until time $t>0$ is arcsine distributed.
\end{theorem}

  Levy's second arcsine law for Brownian motion, Theorem~\ref{thm:levy_arcsine}, has served as a fruitful seedling for various further investigations into the path behaviour of Brownian motion and other stochastic processes and may be regarded as a cornerstone result in probability. For instance, L\'evy also showed that for any $s, t>0$
\begin{align*}
  \frac{A_t}{t} =_d \frac{A_{T_s}}{T_s},
\end{align*}
where $(T_t)_{t\geq 0}$ is the inverse of the continuous local time process $(L_t)_{t\geq 0}$ of Brownian motion at zero defined by $T_s\coloneqq\inf\{ t\geq 0\colon L_t>s  \}$. Put differently, regardless of whether Brownian motion is stopped at some deterministic time $t>0$ where $B_t\neq 0$ a.s.,  or at the random time $T_s$ where $B_{T_s}=0$ a.s., the proportion of time that Brownian motion spends positive is arcsine distributed in both cases. This curious observation inspired deeper investigations in the course of which Jim Pitman and Marc Yor~\cite{PitmanYor1992} found that the lengths of excursions of Brownian motion away from zero are Poisson-Dirichlet distributed when put in non-increasing order. In fact, they derived similar identities for any process whose zero set is the range of a stable subordinator, e.g.~a Bessel process of dimension $d\in (0, 2)$.

L\'evy derived his second arcsine law by an argument involving $(T_t)$ and the fact that this is a stable subordinator with index $1/2$, cf.~\cite{PitmanYor1992}. L\'evy's second arcsine law may also be derived by viewing Brownian motion as the diffusion limit of suitably chosen random walks as justified by Monroe Donsker's invariance principle, cf.~\cite[Proof of Theorem 5.28, p.~139]{MoertersPeres2010}. Yet another derivation may be obtained by the Feynman-Kac formula, cf.~\cite[first application of Theorem 7.43]{MoertersPeres2010}. A recent elementary derivation may be found in~\cite{Pitters2025}. Here is the sketch of a beautiful argument by Pitman and Yor~\cite[Section 4]{PitmanYor2007} employing tools from excursion theory. Denote by
\begin{align*}
  A_t^- \coloneqq \int_0^t \mathbf 1\{ X_s<0  \}ds;\qquad t\geq 0,
\end{align*}
the negative sojourn time of $X$ up until time $t$. Then $(A_{T_t})_{t\geq 0}$ and $(A_{T_t}^-)_{t\geq 0}$ are i.i.d.~(because of the symmetry of $X$, and since positive/negative excursions of Brownian motion are independent); by Kiyosi It\^o's excursion theory they are $(1/2)$-stable subordinators, so
\begin{align*}
  \frac{1}{t}(A_{T_t}, A_{T_t}^-)=_d \left ( \frac{1}{4N^2}, \frac{1}{4N'^2}  \right );\qquad t>0,
\end{align*}
where $N, N'$ are i.i.d.~standard Gaussian. We have
\begin{align}\label{eq:arcsine_law_excursions}
  \frac{A_t}{t} =_d \frac{1/N^2}{1/N^2+1/N'^2} = \frac{N'^2}{N'^2+N^2} =_d \cos^2\Theta;\qquad t>0,
\end{align}
and by~\eqref{def:arcsine} we conclude Levy's second arcsine law for Brownian motion. Mind the ratio of squares of independent Gaussians in~\eqref{eq:arcsine_law_excursions} as we will encounter it in more general form in the representation of the positive sojourn time of L\'evy processes with constant positivity in equation~\eqref{eq:poisson_ratio} of our main result, Theorem~\ref{thm:main}.

\emph{Occupation times of other stochastic processes.} L\'evy~\cite{Levy1939} also showed that if $X$ denotes a standard Brownian bridge, then $A_1\sim\uniform(0, 1)$. Since the publication of L\'evy's second arcsine law for Brownian motion a vast body of literature has been devoted to studying occupation times of other stochastic processes. Let us turn to some examples. Brownian motion with non-zero drift $\mu\in\mathbb R\setminus\{ 0\}$ and its occupation time has been studied in the context of option pricing in mathematical finance, cf.~\cite{Akahori1995, EmbrechtsRogersYor1995}. The distribution function of the occupation time may be found in~\cite[Theorem 1.1]{Akahori1995}, its density in~\cite[equation (4a)]{EmbrechtsRogersYor1995}, and its higher moments in~\cite[Remark 6]{Pitters2025}. To the best of our knowledge Brownian motion with non-zero drift is the only example of a L\'evy process $X$ that does not have constant positivity (we define this notion in (i) of Theorem~\ref{thm:getoor_sharpe}) and for which the density and distribution of the law of its occupation time are known. We provide a second example of a L\'evy process that does not have constant positivity and for which we work out the density, Laplace transform, and double Laplace transform of its occupation time in Section~\ref{sec:one-half-stable} of this article.

In~\cite{PitmanYor1992} the authors showed that if $X$ denotes a (skew) Bessel process of dimension $d\in (0, 2),$ then
\begin{align*}
  \Prob{A_1\in dx} &= \frac{a\sin\pi\alpha}{\pi}\frac{x^\alpha\bar{x}^{\alpha-1}+x^{\alpha-1}\bar{x}^\alpha}{a^2x^{2\alpha}+2a(x\bar x)^\alpha\cos\pi\alpha+\bar{x}^{2\alpha}}\mathbf 1_{(0, 1)}(x)dx,
\end{align*}
where $a\coloneqq (1-p)/p,$ $\alpha\coloneqq 1-d/2,$ and $\bar x\coloneqq 1-x,$ and $p$ denotes the probability of a positive excursion of $X$. This law goes back at least to~\cite{Lamperti1958} where it was discovered as the limiting distribution of occupation times for certain discrete time processes. The special case $\alpha=1/2$ of this density was identified as the proportion of time that skew Brownian motion spends above zero in~\cite{KeilsonWellner1978}. For certain bridges of exchangeable and L\'evy processes the positive sojourn time is uniformly distributed, cf.~\cite{FitzsimmonsGetoor1995, Knight1996}. Further examples of some families of one-dimensional diffusions whose occupation times have been characterised explicitely may be found in~\cite{SalminenStenlund2021}.
For $c\in [0, 1]$ let $\Arcsin(c)$ denote the distribution (not to be confused with the trigonometric function typeset as $\arcsin$)
  \begin{align}
    \begin{cases}
      \text{ with density } \frac{\sin c\pi}{\pi}x^{c-1}(1-x)^{-c}\mathbf 1_{(0, 1)}(x)dx & \text{ if } c\in (0, 1),\\
      \delta_c & \text{ if } c\in \{0, 1\},
    \end{cases}
  \end{align}
  supported in $[0, 1],$ where $\delta_x$ denotes the Dirac delta measure with unit mass in $x\in\mathbb R$. We refer to $\Arcsin(c)$ as the generalised arcsine distribution with parameter $c$. Notice that $\Arcsin(1/2)$ is the standard arcsine distribution with density~\eqref{eq:arcsine_density}. Ronald Getoor and Michael Sharpe~\cite[Theorem 2.7]{GetoorSharpe1994} provided the following remarkable characterization of the occupation time of a L\'evy process that we present in slightly different notation.

\begin{theorem}[Cf.~Theorem 2.7 in~\cite{GetoorSharpe1994}]\label{thm:getoor_sharpe}
Consider the occupation time $A_t=\int_0^t\mathbf 1\{ X_s>0\}ds$ of a L\'evy process $X$. The following are equivalent:
\begin{enumerate}[(i)]
  \item Constant positivity. There exists $c\in [0, 1]$ such that $\Prob{X_t>0}=c$ for all $t>0.$
  \item There exists $c\in [0, 1]$ such that $A_t/t\sim\Arcsin(c)$ for any $t>0$.
  \item There exists $c\in [0, 1]$ such that $q\int_0^\infty e^{-qt}\Prob{X_t>0}dt=c$ for any $q>0$.
\end{enumerate}
If one of these properties holds, the respective constants $c$ agree.
\end{theorem}

If $c=\Prob{X_t>0}=\Prob{X_t<0}=1/2$ for all $t>0,$ then  $X$ is called a symmetric L\'evy process (Brownian motion being the prime example), and $A_t/t\sim\Arcsin(1/2)$ in agreement with L\'evy's second arcsine law for Brownian motion. The special case of Theorem~\ref{thm:getoor_sharpe} when $X$ is a symmetric stable L\'evy process (defined in Section~\ref{sec:one-half-stable}) was shown in~\cite{Kac1951}.

In the case of an arbitrary L\'evy process $X$ the distribution of $A_t$ was recently characterised via its moments in~\cite[Theorem 1.3]{AurzadaDoeringPitters2024}. To present this result, we need some further notation. A partition of a set $S$ is a set, $\pi$ say, of nonempty pairwise disjoint subsets of $S$ whose union is $S$. Let $\card S$ denote the cardinality of $S.$ For some natural number $m$ let $\mathscr P_m$ denote the set of all partitions of $\{1, \ldots, m\}$. For functions $f, g\colon\mathbb R\to\mathbb R$ define $f*g\colon\mathbb R\to\mathbb R$ by $(f*g)(t)\coloneqq\int_0^t f(s)g(t-s)ds$. We call $f*g$ the convolution of $f$ and $g$.

\begin{theorem}[Theorem 1.3 in~\cite{AurzadaDoeringPitters2024}]
  Let $X$ denote an arbitrary L\'evy process started from zero, and let $A_t$ denote its positive sojourn time as defined in~\eqref{def:sojourn_time}. Then the distribution of $A_t$ is uniquely determined by its moments
\begin{align}\label{eq:ot_levy_moments}
  \Ex[A_t^m] &= \sum_{\rho\in\mathscr P_m}\int_0^t \left (\conv_{B\in\rho}f_{\card B}\right )(s) ds;\qquad m\in\mathbb N, t\geq 0,
  \intertext{with}
  f_b(u) &\coloneqq u^{b-1}\Prob{X_u>0}; \quad b\in\mathbb N.
\end{align}
\end{theorem}

In~\cite{AurzadaDoeringPitters2024} the authors develop a sampling approach to characterise the distribution of $A_t$ via its moments for some arbitrary process $X,$ and apply this method to characterise the occupation times of spherical fractional Brownian motion (cf.~\cite[Theorem 1.2]{AurzadaDoeringPitters2024}), L\'evy processes (cf.~\cite[Theorem 1.3]{AurzadaDoeringPitters2024}), and L\'evy bridges (cf.~\cite[Theorem 1.4]{AurzadaDoeringPitters2024}). We also put this simple yet powerful sampling method to use in order to derive our results in Section~\ref{sec:proofs}.

\section{Main results}

In this article we focus on the positive sojourn time in the case that $X$ is a L\'evy process. Recall that $X$ is said to be a L\'evy process issued from the origin if it has the following properties: i) the paths of $X$ are $\mathbb P$-almost surely right-continuous with left limits, ii) $X_0=0$ a.s., iii) for $0\leq s\leq t,$ $X_t-X_s$ is equal in distribution to $X_{t-s}$, iv) for $0\leq s\leq t,$ $X_t-X_s$ is independent of $(X_u, u\leq s)$.

\subsection{Characterization of positive sojourn time of L\'evy processes}\label{sec:levy}

Our first result is a characterization of the positive sojourn time $A_t$ for abritrary L\'evy processes $X$ in terms of an integral formula for the Laplace transform of $t\mapsto \Ex[\exp(-\lambda A_t)]$ in Proposition~\ref{prop:ot_levy_laplace-transform}.

\begin{proposition}\label{prop:ot_levy_laplace-transform}
  Consider the positive sojourn time $A_t=\int_0^t \mathbf 1\{ X_s>0  \}ds$ up until time $t\geq 0$ of some L\'evy process $(X_t)_{t\geq 0}.$ For $q,\lambda>0$ the Laplace transform of $A_t$ is uniquely determined by
  \begin{align}\label{eq:ot_levy_laplace-transform}
    G(q, \lambda)\coloneqq \Ex\left [\int_0^\infty e^{-qt}\exp(-\lambda A_t)dt\right ] &= q^{-1}\exp\left (-\int_0^\infty e^{-qt}t^{-1}\Prob{X_t>0} (1-e^{-\lambda t}) dt\right ).
  \end{align}
\end{proposition}

As a Corollary we recover the generalised arcsine law for L\'evy processes with constant positivity due to Getoor and Sharpe~\cite[Theorem 2.7]{GetoorSharpe1994}. We say that a real random variable is gamma distributed with scale parameter $q>0$ and shape parameter $k>0$ (or: with parameters $q, k$ for short), if its distribution has density $t\mapsto \mathbf 1_{(0, \infty)}(t)q^kt^{k-1}e^{-qt}/\Gamma(k)$. We denote the gamma distribution with parameters $q,k$ as $\Gam(q, k),$ and we write $\Exp(q)=\Gam(q, 1)$ for an exponential distribution with rate parameter $q$.

\begin{corollary}\label{cor:main}
  Consider the setting of Proposition~\ref{prop:ot_levy_laplace-transform}. Then the following statements are equivalent:
\begin{enumerate}[(i)]
  \item Constant positivity. There exists some $c\in [0, 1]$ such that $\Prob{X_t>0}=c$ for all $t>0$.
  \item There exists $c\in [0, 1]$ such that for any $t>0$ we have $A_t/t\sim\Arcsin(c)$.
  \item There exists $c\in [0, 1]$ such that $q\int_0^\infty e^{-qt}\Prob{X_t>0}dt=c$ for any $q>0$.
  \item $G(q, \lambda)=1/q(1+\lambda/q)^{c}$ for all $q, \lambda>0$ and some $c\in [0, 1]$.
  \item If $E^{(q)}\sim\Exp(q)$ is independent of $X$, then $A_{E^{(q)}}\sim\Gam(q,  c)$ for some $c\in [0, 1]$, the r.v.s $E^{(q)}, A_{E^{(q)}}/E^{(q)}$ are independent, and
  \begin{align}\label{eq:ratio}
    \frac{A_{E^{(q)}}}{E^{(q)}}\sim\Arcsin(c).
  \end{align}
\end{enumerate}
If one of the above properties is met, the respective constants $c$ agree.
\end{corollary}
The equivalence of (i), (ii) and (iii) in Corollary~\ref{cor:main} goes back to Getoor and Sharpe, cf.~Theorem~\ref{thm:getoor_sharpe}. The function $G(q, \lambda)$ in (iv) featured in their proof of~\cite[Theorem 2.7]{GetoorSharpe1994}.

\begin{remark}
  If $X$ is Brownian motion, then $A_1=_d t^{-1}A_t$ for any $t>0$ is an immediate consequence of the scaling property. In fact, this argument applies to any scale-invariant L\'evy process.
\end{remark}

As a consequence of Proposition~\ref{prop:ot_levy_laplace-transform} we obtain our main result, a representation of the occupation time up until an exponential time independent of $X$ as the sum of points in a Poisson process, Theorem~\ref{thm:main}.

\begin{theorem}[Poisson representation of positive sojourn time of L\'evy process]\label{thm:main}
  Consider the positive sojourn time $A_t=\int_0^t \mathbf 1\{ X_s>0  \}ds$ up until time $t>0$ of some L\'evy process $(X_t)_{t\geq 0}.$ For some $q>0$ let $E^{(q)}\sim\Exp(q)$ be independent of $X$. Then we have the representation in law
  \begin{align}\label{eq:main_identity}
    A_{E^{(q)}} =_d  \sum_{T\in \Pi} T,
  \end{align}
  where $\Pi$ is a Poisson process on the positive half-line with intensity $e^{-qt}t^{-1}\Prob{X_t>0}$. In particular, the mean and the variance of $A_{E^{(q)}}$ are given by
  \begin{align}\label{eq:main_mean_variance}
    \Ex\left [A_{E^{(q)}}\right ] = \int_0^\infty \Prob{X_t>0}e^{-qt}dt\leq q^{-1},\qquad \Var\left (A_{E^{(q)}}\right )=\int_0^\infty \Prob{X_t>0}te^{-qt}dt\leq 1.
  \end{align}
  
  If $X$ has constant positivity, i.e.~$\Prob{X_t>0}=c$ for some $c\in [0,1]$ and all $t>0,$ then colour each point $T$ of the Poisson process $\bar\Pi$ on the positive half-life with intensity $e^{-qt}t^{-1}$ black with probability $c$ and white otherwise, the colour being independent of $X$ and of the colours of other points. Then
      $$\Gamma\coloneqq \sum_{T\in \bar\Pi, T\text{ is black}} T\sim\Gam(q, c),\qquad \Gamma'\coloneqq\sum_{T\in \bar\Pi, T\text{ is white}} T\sim\Gam(q, 1-c),$$
  are independent, and
  \begin{align}\label{eq:poisson_ratio}
    \frac{A_t}{t} =_d \frac{\Gamma}{\Gamma+\Gamma'}\sim\Arcsin(c).
  \end{align}
\end{theorem}

\emph{Some observations about and questions raised by Theorem~\ref{thm:main}.}
\begin{itemize}
  \item If $X$ is a symmetric L\'evy process ($c=1/2$) we have from~\eqref{eq:poisson_ratio}
  $$ \frac{A_t}{t} =_d \frac{\Gamma_{\frac 1 2}}{\Gamma_{\frac 1 2}+\Gamma'_{\frac 1 2}} =_d \frac{N'^2}{N'^2+N^2}, $$
    and we recover the ratio of independent standard Gaussians $N, N'$ in the derivation of L\'evy's second arcsine law for Brownian motion by Pitman and Yor, cf.~\eqref{eq:arcsine_law_excursions}, since $N^2\sim\Gam(1/2, 1/2)$.
  \item Take $X$ to be Brownian motion. This process has continuous sample paths, in particular the integral $A_{E^{(q)}}=\int_0^{E^{(q)}} \mathbf 1\{ X_s>0  \}ds$ is a continuous r.v., but the sum on the right hand side of~\eqref{eq:main_identity} is a rather discrete-looking object. How can this be reconciled?
  \item Consider a L\'evy process $X$ with constant positivity, i.e.~$\Prob{X_t>0}=c$ for some $c\in (0, 1)$ and any $t>0.$ The sum $\sum_{T\in \Pi} T$ is almost surely finite as we saw in Corollary~\ref{cor:main}, and the random vector constructed by ordering the relative sizes
  $$ \frac{T}{\sum_{T\in \Pi} T}; \quad T\in\Pi,  $$
  of the points $T$ of $\Pi$ in non-increasing order is Poisson-Dirichlet distributed with parameter $c$. In fact, this is one construction (of many) of the Poisson-Dirichlet distribution, and due to John Kingman, cf.~\cite[Chapter 9]{Kingman1993}. At first sight it seems rather surprising that the Poisson-Dirichlet distribution appears so naturally and in such a general setting of L\'evy processes, which leads us to the next question.
  \item Is there an almost sure identity underlying representation~\eqref{eq:main_identity}? It is tempting to think that the summands on the right hand side of~\eqref{eq:main_identity} are (related to) the lengths of positive excursions of the process $X$ (up until time $E^{(q)}$), provided that the classical excursion theory of It\^o (cf.~\cite{PitmanYor2007}) may be applied to $X.$ However, It\^o's excursion theory crucially relies on the notion of a local time. What about L\'evy processes $X$ that do not have a local time and therefore cannot be studied by It\^o's excursion theory? Clearly, an almost sure interpretation of~\eqref{eq:main_identity} in its full generality in terms of ``excursions'' of $X$ would require an extension of It\^o's excursion theory.
\end{itemize}

We plan to investigate above questions in forthcoming work. To the best of our knowledge, i) the representation of the positive sojourn time as a sum of points in a Poisson process in Theorem~\ref{thm:main}, ii) the characterization of the positive sojourn time via its double Laplace transform in Proposition~\ref{prop:ot_levy_laplace-transform} together with iii) the characterization of $A_t$ via its moments, Theorem 1.3 in~\cite{AurzadaDoeringPitters2024}, are the only explicit characterizations of the law of $A_t$ known so far in the general case of an arbitrary L\'evy process.

We now apply the integral formula of Proposition~\ref{prop:ot_levy_laplace-transform} to the strictly stable subordinator with index of stability $1/2$ and drift to obtain both the Laplace transform of $A_t$ and the density of its distribution in Theorem~\ref{thm:ot_one-half-stable}.

\subsection{Positive sojourn time of $(1/2)$-stable subordinator with drift}\label{sec:one-half-stable}
We now turn to a specific example of a L\'evy process, the $(1/2)$-stable subordinator with drift, that we will define shortly. Studying the distribution of the positive sojourn time of this process is particularly interesting. Its positivity function
is known explicitely (cf.~equation~\eqref{eq:subordinator_positivity_function}); however, this positivity function is not constant, it is therefore not covered by the characterization in Corollary~\ref{cor:main}, and therefore it cannot be arcsine distributed. Moreover, despite having an analytic expression for this positivity function, cf.~equation~\eqref{eq:subordinator_positivity_function}, we were not able to work out the moments of the occupation time via Theorem~\ref{thm:levy_moments}, as the convolutions in equation~\eqref{eq:ot_levy_moments} appear to be somewhat convolved. It turns out, however, that in this case Theorem~\ref{prop:ot_levy_laplace-transform} lends itself better for calculations and is the more elegant way to proceed.

Before we define the $(1/2)$-stable subordinator, we recall some basic notions of stable distributions and L\'evy processes mainly following~\cite{Janson2022}. We say that the distribution of a non-degenerate random variable $X$ has a stable distribution if there exist constants $a_n>0,$ $b_n$ such that, for any $n\geq 1,$ if $X_1, X_2, \ldots$ are i.i.d.~copies of $X$ and $S_n\coloneqq\sum_{k=1}^n X_k,$ then $S_n=_d a_nX+b_n$. The distribution is strictly stable iff $b_n=0$. We call the r.v.~$X$ (strictly) stable if its distribution is. The norming constants $a_n$ are necessarily of the form $a_n=n^{1/\alpha}$ for some $\alpha\in (0, 2],$ and $\alpha$ is called the index or characteristic exponent of the distribution. We call a distribution (or random variable) $\alpha$-stable if it is stable with index $\alpha.$ For equivalent definitions of and more information on stable distributions see~\cite{SamarodnitskyTaqqu1994, KyprianouPardo2022, Janson2022}. If the L\'evy process $(X_t)_{t\geq 0}$ has non-decreasing paths it is called a subordinator. We call $(X_t)_{t\geq 0}$ a strictly stable L\'evy process if its marginal distributions at each fixed time are non-Gaussian and strictly stable. Equivalently, a strictly stable L\'evy process is a L\'evy process $(X_t)$ for which there exists an $\alpha>0$ such that, for any $c>0,$ $(cX_{c^{-\alpha}t})_{t\geq 0}$ is equal in distribution to $(X_t)$; cf.~\cite[Chapter 3]{KyprianouPardo2022} and notice that when the authors refer to a stable random variable/L\'evy process they mean what we call `strictly stable'.

Let $(S_t)_{t\geq 0}$ denote the strictly stable subordinator with index of stability $1/2$. This process may be constructed as the first hitting time of one-dimensional standard Brownian motion $(B_t),$ i.e.
\begin{align}
  S_t \coloneqq \inf\left \{s>0\colon B_s=\frac{t}{\sqrt 2}  \right\}.
\end{align}
For a real number $\mu$ fixed arbitrarily define the strictly stable subordinator $S^{(\mu)}\coloneqq \{ S^{(\mu)}_t, t\geq 0  \}$ with index of stability $1/2$ and drift $\mu$ by setting $S^{(\mu)}_t\coloneqq S_t+\mu t$, and denote its positive sojourn time
\begin{align}
  A_t^{(\mu)}\coloneqq \int_0^t \mathbf 1_{(0, \infty)}(S^{(\mu)}_s)ds; \qquad t\geq 0.
\end{align}
Since for $\mu\geq 0$ we have $A^{(\mu)}_t=1$ a.s.~for $t>0,$ we henceforth focus on the case of negative drift $-\mu<0$. (For $\mu>0$ it is obvious that $A^{(\mu)}_t=1$ a.s.. To see this for $\mu=0,$ it suffices to show that the law of $A^{(0)}_t$ has no atom in zero, since $S^{(0)}$ is a subordinator. $S^{(0)}$ is a strictly stable L\'evy process and thus has constant positivity: $\Prob{S_t>0}=c>0$ for some $c>0$ and any $t>0.$ By Corollary~\ref{cor:main}, $A_t^{(0)}/t$ has distribution $\Arcsin(c)$, hence no atom at zero.)

Whenever it appears more convenient, we work with the error function instead of the Gaussian distribution function. Recall that the error function is defined as, cf.~\cite[7.2.1]{OlverLozierBoisvertClark2010},
\begin{align}
  \erf(z) &\coloneqq \frac{2}{\sqrt{\pi}}\int_0^z e^{-t^2}dt; \qquad z\in\mathbb C.
\end{align}
We can now state our result on the occupation time of the $(1/2)$-stable subordinator with drift.

\begin{theorem}\label{thm:ot_one-half-stable}
  Fix $t>0$. The occupation time $A_t^{(-\mu)}$ of the strictly stable subordinator $S^{(-\mu)}$ with index of stability $1/2$ and drift $-\mu<0$ has density, Laplace transform and double Laplace transform
\begin{align}
  x\mapsto & h(x)g(t-x)\mathbf 1_{(0, t)}(x),\\
  \Ex\left [\exp(-\lambda A_t^{(-\mu)})\right ] &=  \int_0^t e^{-\lambda u}h(u)g(t-u)du,\\
  \Ex\left [\int_0^\infty e^{-qt}\exp(-\lambda A_t^{(-\mu)})dt\right ] &= \frac{1}{q+\lambda} \left ( \frac{\sqrt b+\sqrt{q+\lambda+b}}{\sqrt b+\sqrt{q+b}} \right )^2,
\end{align}
respectively, where $b\coloneqq 1/(4\mu),$ and
\begin{align*}
	g(u) &\coloneqq (2bu+1)(\erf\sqrt{bu}-1)+2e^{-bu}\sqrt{\frac{bu}{\pi}},\\
	h(u) &\coloneqq 2\left (b(\erf\sqrt{bu}+1) +\sqrt{\frac{b}{\pi u}}e^{-bu}  \right ).
\end{align*}
\end{theorem}
  In principle the law of $A^{(-\mu)}_t$ could be characterised by computing its moments, for instance via~\eqref{eq:ot_levy_moments}. However, with the specific positivity function~\eqref{eq:subordinator_positivity_function} of the $(1/2)$-stable subordinator we were not able to work out explicitely the integral in~\eqref{eq:ot_levy_moments}. Instead, the integral formula for the double Laplace transform of $A^{(-\mu)}_t$ in Proposition~\ref{prop:ot_levy_laplace-transform} turns out to be more conducive for calculations.

\section{Proofs}\label{sec:proofs}
Before we turn to the proofs of our results, we provide an overview of the ideas and tools that we employ.

\subsection*{Preliminaries}
 One idea guiding our approach is what we call the sampling of the occupation time as introduced in~\cite{AurzadaDoeringPitters2024}. The simplest instance of this method may be phrased as follows.

\begin{proposition}[Sampling the occupation time]
 Consider a real stochastic process $(X_t)_{t\geq 0}$. The distribution of its occupation time $A_t\coloneqq\int_0^t\mathbf 1\{  X_s>0  \}ds$ may be characterised via its moments by sampling the process $(X_t)$ at times $U_1, U_2, \ldots,$ which are i.i.d.~with $U_1\sim\uniform(0, t),$ namely
\begin{align}\label{eq:sampling}
  \Ex\left [ \left ( \frac{A_t}{t}   \right )^m  \right ] = \Prob{X_{U_k}>0, 1\leq k\leq m}; \qquad m\geq 1.
\end{align}
\end{proposition}
\begin{proof}
  Write $F_{A_t/t}$ for the cumulative distribution function of $A_t/t$ defined by $F_{A_t/t}(x)\coloneqq\Prob{A_t/t\leq x},$ $x\in\mathbb R$. Conditionally given $A_t=\int_0^t\mathbf 1\{  X_s>0  \}ds,$ $(\mathbf 1\{X_{U_k}>0\})_{k\geq 1}$ is a sequence of independent Bernoulli trials with success probability $A_t/t.$ Hence
  \begin{align*}
    \Prob{X_{U_k}>0, 1\leq k\leq m} = \int_0^1 p^m F_{A_t/t}(dp) = \Ex\left [\left (\frac{A_t}{t}\right )^m\right ]; \qquad m\geq 1.
  \end{align*}
Since $A_t$ is bounded, its distribution is uniquely determined by its moment sequence according to the Hausdorff's moment problem, cf.~\cite[Chapter III]{Widder1941}.
\end{proof}
If $(X_t)$ is a L\'evy process,~\cite[Theorem 1.3]{AurzadaDoeringPitters2024} shows that the sampling of the occupation time~\eqref{eq:sampling} leads to the integral formula~\eqref{eq:ot_levy_moments} with integrands that are various convolutions of functions of the form $t\mapsto t^b\Prob{X_t>0}$. Though this formula holds in full generality, it may or may not be tractable depending on the specific form of the positivity function $\Prob{X_t>0}$ at hand. In working on specific examples of positivity functions we found that considering the Laplace transform of $t\mapsto \Ex[(-\lambda A_t)]$ may be a viable alternative to characterising the distribution of $A_t$. Interestingly, it turns out that this Laplace transform may again be represented in terms of the survival probabilities $\Prob{X_{U_k}>0, 1\leq k\leq m}$. Moreover, if $(X_t)$ is a L\'evy process, these survival probabilities lead to fluctuations of random walks and Spitzer's combinatorial Lemma more specifically. As we have shown earlier, in our setting it is convenient to rephrase Spitzer's combinatorial Lemma in terms of (complete) Bell polynomials.

Let us now recall the notion of (complete and partial) Bell polynomials and some basic facts about their generating functions taken from the exposition in~\cite[Chapter 1]{Pitman2006}. For any two fixed sequences $v_\bullet\coloneqq (v_k)_{k\geq 1}$ and $w_\bullet\coloneqq (w_k)_{k\geq 1}$ define the complete Bell polynomials associated with $(v_\bullet, w_\bullet)$ by
\begin{align}
  B_k(v_\bullet, w_\bullet) \coloneqq \sum_{l=1}^k v_lB_{kl}(w_\bullet);\qquad k\in\mathbb N,
\end{align}
where
\begin{align}
  B_{kl}(w_\bullet)\coloneqq \sum_{\rho\in\mathscr P_{kl}}\prod_{B\in\rho} w_{\card B};\qquad 1\leq l\leq k,
\end{align}
denotes the $(k,l)$th partial Bell polynomial (associated with $w_\bullet$), and $\mathscr P_{kl}$ is the set of all partitions of $\{1, \ldots, k\}$ containing $l$ blocks. In addition, we write $\mathscr P_k$ for the set of all partitions of $\{1, \ldots, k\}$. Recall that for some sequence $(a_k)_{k\geq 1}$ of real numbers its exponential generating function is defined as the formal power series $\sum_{k\geq 1}a_kx^k/k!$. In what follows we work with the exponential generating function of the complete Bell polynomial which is well known to be given by
\begin{align}
  \sum_{k\geq 1} B_k(v_\bullet, w_\bullet)\frac{x^k}{k!} = v(w(x)),
\end{align}
whenever either of these quantities is well-defined, with $v, w$ the exponential generating functions of $v_\bullet, w_\bullet,$ respectively. We may now recast a variant of Spitzer's combinatorial Lemma, a cornerstone result in the fluctuation theory of random walks, in terms of Bell polynomials as follows.

\begin{proposition}[cf.~Corollary 3 in~\cite{AurzadaDoeringPitters2024}]\label{prop:spitzers_lemma_bell}
  Let $(S_n)_{n\geq 1}$ denote a random walk defined by $S_n\coloneqq\xi_1+\cdots+\xi_n,$ where $\xi_1, \xi_2, \ldots$ is a sequence of real-valued i.i.d.~random variables. The survival probability of $(S_n)$ is given by
  \begin{align}\label{eq:spitzers_lemma}
    \Prob{S_k\geq 0, 1\leq k\leq m} = \frac{1}{k!}\sum_{\rho\in\mathscr P_k} \prod_{B\in\rho} (\card B-1)!\Prob{S_{\card B}\geq 0}.
  \end{align}
  The identity in~\eqref{eq:spitzers_lemma} still holds if the weak inequalities are replaced by their strict versions.
\end{proposition}
Notice that the right hand side in~\eqref{eq:spitzers_lemma} is a complete Bell polynomial.

\subsection*{Proof of Theorem~\ref{thm:main}}
We are now ready to prove Proposition~\ref{prop:ot_levy_laplace-transform}, Corollary~\ref{cor:main} and Theorem~\ref{thm:main}.

\begin{proof}[Proof of Proposition~\ref{prop:ot_levy_laplace-transform}]
  For $q,\lambda> 0$ we have $G(q, \lambda)=\int_0^\infty e^{-qt}\Ex[\exp(-\lambda A_t)]\dd t$ by the Fubini-Tonelli theorem. By standard properties of the Stieltjes integral, setting $F_t(\lambda)\coloneqq \int_0^t \Ex[\exp(-\lambda A_s)] ds$ we have $G(q, \lambda)=\int_0^\infty e^{-qt}dF_t(\lambda)$. Since $F_t(\lambda)$ is non-decreasing in $t$, the Laplace transform $q\mapsto G(q, \lambda)$ uniquely determines $t\mapsto\Ex[\exp(-\lambda A_t)]$ for almost all values of $t$, cf.~\cite[Theorem 6.3]{Widder1941}.

Let us first recall a formula for the Laplace transform of the $m$th moment of $A_t$ that was used in the context of sampling of the occupation time in~\cite[Equation (14)]{AurzadaDoeringPitters2024}:
\begin{align}\label{eq:moment_lt}
  \int_0^\infty e^{-qt}\Ex[A_t^m]dt = \frac{m!}{q^{m+1}}\Prob{X_{T^{(q)}_k}>0, 1\leq k\leq m}; \qquad q>0,
\end{align}
where $0<T^{(q)}_1<T^{(q)}_2<\cdots$ denote the arrival times in a Poisson process on the half-line with intensity $q>0$ independent of $(X_t).$ Notice that for fixed $\lambda, t$ we have almost sure convergence $\sum_{k=0}^n (-\lambda A_t)^k/k!\to\exp(-\lambda A_t)$ as $n\to\infty$. By Lebesgue's dominated convergence theorem, the functions $f_n\colon [0, \infty)\to\mathbb R$ defined by $f_n(t)\coloneqq e^{-qt}\Ex[\sum_{k=0}^n (-\lambda A_t)^k/k!]$ converge pointwise to $f(t)\coloneqq e^{-qt}\Ex[\exp(-\lambda A_t)]$ as $n\to\infty$. Again applying Lebesgue's dominated convergence theorem, $\int_0^\infty f_n(t)dt\to \int_0^\infty f(t)dt$ as $n\to\infty$.

Multiplying~\eqref{eq:moment_lt} by $(-\lambda)^m/m!$ and summing over $m\geq 0,$ we obtain
\begin{align}
  G(q, \lambda) = \Ex\left [ \int_0^\infty e^{-qt}\exp(-\lambda A_t)dt\right ] = \frac 1 q \sum_{k=0}^\infty p_k(q)\left (-\frac{\lambda}{q} \right )^k,
\end{align}
where $p_0(q)\coloneqq 1$, and $p_k(q)\coloneqq \Prob{X^{(q)}_{T_1}>0, \ldots, X^{(q)}_{T_k}>0},$ $k\geq 1,$ are the persistence probabilities. By Proposition~\ref{prop:spitzers_lemma_bell},
\begin{align}
  p_k(q) &= k!^{-1}B_k\left (1^\bullet, (\bullet-1)! \Prob{X^{(q)}_{T_{\bullet}}>0}  \right ).
\end{align}
Consequently, $qG(q, \lambda)-1$ is the exponential generating function of a complete Bell polynomial. More specifically,
\begin{align}
  G(q, \lambda) &= q^{-1}\left (v\left (w\left (-\frac \lambda q \right )\right )+1\right ),
\end{align}
where $v, w$ denotes the exponential generating function of $1^\bullet$ and $(\bullet-1)! \Prob{X^{(q)}_{T_{\bullet}}>0},$ respectively. In particular, $v(\zeta)=e^\zeta-1.$ Since $T_k^{(q)}\sim\Gam(q, k),$ we find
\begin{align*}
  w(\zeta) &= \sum_{k\geq 1} (k-1)!\Prob{X_{T_k^{(q)}}>0}\frac{\zeta^k}{k!}\\
  &= \sum_{k\geq 1}\frac{\zeta^k}{k} \int_0^\infty \frac{q^k}{\Gamma(k)}t^{k-1}e^{-qt}p_t dt\\
  &=  \int_0^\infty e^{-qt}\frac{p_t}{t} \sum_{k\geq 1}\frac{(\zeta q t)^k}{k!} dt  = \int_0^\infty e^{-qt}\frac{p_t}{t} (e^{\zeta q t}-1) dt,
\end{align*}
where in the penultimate step we applied the dominated convergence theorem. Putting everything together we obtain the claim.
\end{proof}

\begin{proof}[Proof of Corollary~\ref{cor:main}]
To derive the claim, our strategy is to show that
  \begin{enumerate}
    \item (i) and (iv) are equivalent,
    \item (v) and (ii) are equivalent,
    \item (v) implies (iv),
    \item (iv) implies (ii), and
    \item (i) and (iii) are equivalent.
  \end{enumerate}
Let $\mathscr A$ denote a random variable with $\Arcsin(c)$ distribution.

1.~(i) and (iv) are equivalent. Suppose (i) holds. By~\eqref{eq:ot_levy_laplace-transform}
  \begin{align*}
    G(q, \lambda) &= q^{-1}\exp\left (-c\int_0^\infty e^{-qt}t^{-1}(1-e^{-\lambda t}) dt\right ) = \frac{1}{q(1+\lambda/q)^c}; \qquad q, \lambda>0,
  \end{align*}
where we employed the Frullani integral $\int_0^\infty e^{-qt}t^{-1}(1-e^{-\lambda t}) dt=\log(1+\lambda/q)$ in the last equality. If (iv) holds, then for any $q>0$ the function $q\mapsto G(q, \lambda)$ uniquely determines the Laplace transform of $t\mapsto t^{-1}\Prob{X_t>0}(1-e^{-\lambda t})$, which in turn uniquely determines $t\mapsto\Prob{X_t>0}$ for Lebesgue almost all $t>0,$ cf.~\cite[Theorem 6.3]{Widder1941}.

2.~(ii) and (v) are equivalent. Suppose that (v) holds, in particular $A_{E^{(q)}}/E^{(q)}=_d \mathscr A$ for any $q>0$. Then for $q, \lambda >0,$ using Fubini-Tonelli,
  \begin{align*}
   \int_0^\infty qe^{-qt}\Ex\left [\exp(-\lambda\mathscr A) \right ]dt = \Ex\left [\exp(-\lambda \mathscr A) \right ] = \Ex\left [\exp(-\lambda A_{E^{(q)}}/E^{(q)}  ) \right ] = \int_0^\infty qe^{-qt}\Ex\left [\exp(-\lambda A_t/t  ) \right ]dt,
  \end{align*}
  and by the uniqueness of the Laplace transform, cf.~\cite[Theorem 6.3]{Widder1941}, we obtain for any $\lambda>0$, first for Lebesgue almost every $t>0,$
  $$ \Ex\left [\exp(-\lambda\mathscr A) \right ] = \Ex\left [\exp(-\lambda A_t/t  ) \right ], $$
 and this identiy in fact holds for all $t>0,$ since $t\mapsto \Ex\left [\exp(-\lambda A_t/t  ) \right ]$ is a continuous function. Thus we just showed that (v) implies (ii). Suppose that (ii) holds. Then for any $E^{(q)}\sim\Exp(q)$ we have
$
   A_{E^{(q)}}/E^{(q)} =_d \mathscr A.
$
 Moreover, $E^{(q)}, A_{E^{(q)}}/E^{(q)}$ are independent, since for two integers $m, n\geq 1$ fixed arbitrarily
\begin{align*}
   \Ex \left [\left (\frac{A_{E^{(q)}}}{E^{(q)}}\right )^m (E^{(q)})^n\right ] = \Ex\left [ \Exc{\left (\frac{A_{E^{(q)}}}{E^{(q)}}\right )^m  }{E^{(q)}} (E^{(q)})^n \right ] = \Ex \left [\mathscr A^m \right ]\Ex \left [ (E^{(q)})^n\right ],
 \end{align*} 
 and the last equality follows from our assumption. Using this independence and writing $A_{E^{(q)}} = E^{(q)}A_{E^{(q)}}/E^{(q)}$, a simple moment computation shows that $A_{E^{(q)}}\sim\Gam(q, c)$, and we obtain iii).

 3.~(v) implies (iv). Assume iii). In particular, for any $q>0$ we have $A_{E^{(q)}}\sim\Gam(q, c),$ thus
 \begin{align*}
   qG(q, \lambda) = \Ex[\exp(-\lambda A_{E^{(q)}})] = \left ( \frac{q}{\lambda+q}  \right )^c; \qquad q, \lambda>0,
 \end{align*}
 and we obtain (iv).

 4.~(iv) implies (ii). Assume (iv) and fix $r>0$ arbitrarily. Setting $\varphi(t)\coloneqq rt,$ the Laplace transform of $t\mapsto \Ex[\exp(\frac 1 r A_{rt})]$ is uniquely determined by
 \begin{align*}
    \Ex\left [ \int_0^\infty e^{-qt}\exp(-\lambda\frac 1 r A_{rt})dt\right ] &= \frac 1 r\Ex\left [ \int_0^\infty e^{-\frac{q}{r}\varphi(t)}\exp(-\frac\lambda r A_{\varphi(t)}) \varphi'(t)  dt\right ],
    \intertext{where $\varphi'$ denotes the derivative of $\varphi,$ and integrating by substitution,}
  &= \frac 1 r\Ex\left [ \int_0^\infty e^{-\frac{q}{r}t}\exp(-\frac \lambda r A_{t})   dt\right ]\\
  &= \frac 1 r G\left (\frac q r, \frac  \lambda r\right ) = \frac 1 r \left (\frac{q/r}{q/r+\lambda/r}\right )^c\frac{1}{q/r} = G(q, \lambda),
 \end{align*}
 where the ultimate and penultimate equalities follow from the assumption. With the same argument as in the proof of Proposition~\ref{prop:ot_levy_laplace-transform} we obtain for any $r, t>0$ that $A_{rt}/r=_d A_t$. Choosing $t=1$ shows (ii).

 5.~Clearly, (i) implies (iii). Suppose (iii) holds. By the uniqueness theorem for Laplace transform, cf.~\cite[Theorem 6.3]{Widder1941}, $\Prob{X_t>0}=c$ for Lebesgue almost all $t>0.$ It is well-known that $t\mapsto\Prob{X_t>0}$ is either left continuous or right continuous on $[0, \infty),$ cf.~\cite[Lemma 2.6]{GetoorSharpe1994}, and therefore (i) holds.
\end{proof}


To derive our main result we make use of Campbell's Theorem. Let us restate Campbell's Theorem for the reader's convenience without proof following~\cite[Section 3.2]{Kingman1993}.

\begin{theorem}[Campbell's Theorem]\label{thm:campbell}
  Let $\Pi$ be a Poisson process on $S$ with mean measure $\mu$, and let $f\colon S\to\mathbb R$ be measurable. Then the sum
  \begin{align}
    \Sigma\coloneqq \sum_{X\in \Pi}f(X)
  \end{align}
  is absolutely convergent with probability if and only if
  \begin{align}
    \int_S \min(\abs{f(x)}, 1)\mu(dx)<\infty.
  \end{align}
If this condition holds, then
\begin{align}\label{eq:exponential_functional}
  \Ex[\exp(\theta \Sigma)] &= \exp\left ( \int_S (e^{\theta f(x)}-1)  \mu(dx)  \right )
\end{align}
for any complex $\theta$ for which the integral on the right converges, and in particular when $\theta$ is pure imaginary. Moreover\begin{align}\label{eq:campbell_mean}
  \Ex[\Sigma] = \int_S f(x)\mu(dx)
\end{align}
in the sense that the expectation exists if and only if the integral converges and they are then equal. If~\eqref{eq:campbell_mean} converges, then
\begin{align}
  \Var(\Sigma) &= \int_S f(x)^2\mu(dx),
\end{align}
finite or infinite.
\end{theorem}

\begin{proof}[Proof of Theorem~\ref{thm:main}]
Using Proposition~\ref{prop:ot_levy_laplace-transform} we find
  \begin{align}\label{eq:exponential_functional}
    \Ex[\exp(-\lambda A_{E^{(q)}})] &= \Ex\left [\int_0^\infty \exp(-\lambda A_t) qe^{-qt} dt \right ]  \\\notag
      &= \exp\left (-\int_0^\infty e^{-qt}t^{-1} \Prob{X_t>0} (1-e^{-\lambda t})  dt \right ) \\\notag
      &= \exp\left (-\int_0^\infty\int_0^\infty  (1-e^{-\lambda t}) \mu(dt) \right ),
  \end{align}
  where $\mu$ is a finite measure on the positive half-line with density $e^{-qt}t^{-1}\Prob{X_t>0}$. Define a Poisson process $\Pi$ on the positive half-line with mean measure $\mu.$ Since $\int_{[0, \infty)} \min(t, 1)\mu(dt)=\int_0^1 e^{-qt}\Prob{X_t>0}dt+\int_1^\infty t^{-1}e^{-qt}\Prob{X_t>0}dt<\infty$ we obtain from Campbell's Theorem, Theorem~\ref{thm:campbell}, for the special case $f(x)\coloneqq x,$
  \begin{align*}
    \Ex\left [\exp\left (-\lambda\sum_{T\in \Pi} T\right ) \right ] = \exp\left ( \int_0^\infty\int_0^\infty (e^{-\lambda t}-1)\mu(dt)   \right).
  \end{align*}
Comparing this formula with the expressions in~\eqref{eq:exponential_functional} we obtain~\eqref{eq:main_identity}. Campbell's theorem now yields~\eqref{eq:main_mean_variance}.

Assume now that $X$ has constant positivy $\Prob{X_t>0}=c$ for $t>0$. Since $\bar\Pi$ is the superposition of the two independent Poisson processes of its black and white points, $\{  T\in\Pi: T\text{ is black}\}$ and $\{  T\in\Pi: T\text{ is white}\}$, the $\Gamma, \Gamma'$ are independent. By Corollary~\ref{cor:main}, (v) we have $\Gamma\sim\Gam(q, c)$ and $\Gamma'\sim\Gam(q, 1-c)$. By Corollary~\ref{cor:main},
  \begin{align}
    \frac{A_t}{t} =_d \frac{A_{E^{(q)}}}{E^{(q)}} =_d \frac{\sum_{T\in \bar\Pi, T \text{ is black}} T}{\sum_{T\in \bar\Pi} T} = \frac{\Gamma}{\Gamma+\Gamma'}\sim\Arcsin(c),
  \end{align}
  where the last identity follows from a well-known representation of the beta distribution.
\end{proof}


\subsection*{Proof of Theorem~\ref{thm:ot_one-half-stable}}
We now turn to the proof of Theorem~\ref{thm:ot_one-half-stable}.

\begin{proof}[Proof of Theorem~\ref{thm:ot_one-half-stable}]

Recall (cf.~\cite[Example 6.3]{Janson2022}) that the density and distribution function of the marginal $S_t$ (supported on $(0, \infty)$ and) is given by
\begin{align*}
  x & \mapsto \frac{t}{\sqrt{2\pi x^3}} e^{-t^2/(4x)}\mathbf 1\{ x>0 \},\\
  \intertext{and}
\Prob{S_t\leq x} &= \begin{cases}
  2\left (1-\Phi\left (\frac{t}{\sqrt{x}}\right ) \right ) & \text{if }x> 0,  \\
  0 & \text{if }x\leq 0,
\end{cases}
\end{align*}
respectively, where $\Phi(x)\coloneqq \frac{1}{\sqrt{2\pi}}\int_{-\infty}^x e^{-y^2/2}dy$ denotes the distribution function of the standard Gaussian distribution. Fix $\mu> 0$ and consider the $(1/2)$-stable subordinator with negative drift, $(S_t^{(-\mu)})_{t\geq 0},$ defined by $S^{(-\mu)}_t\coloneqq S_t-\mu t.$ In what follows we work with the error function instead of $\Phi,$ as this turns out to be more convenient. Recall that the error function is defined as, cf.~\cite[7.2.1]{OlverLozierBoisvertClark2010},
\begin{align}
  \erf(z) &\coloneqq \frac{2}{\sqrt{\pi}}\int_0^z e^{-t^2}dt; \qquad z\in\mathbb C.
\end{align}
Thus for $x>0,$ using $\Phi(0)=1/2$ and integration by substitution
\begin{align*}
  2\Phi(x)-1 = 2(\Phi(x)-\Phi(0))   = \sqrt{\frac{2}{\pi}}\int_0^x e^{-y^2/2}dy = \frac{2}{\sqrt\pi} \int_0^{x/\sqrt 2} e^{-y^2}dy=\erf\left (\frac{x}{\sqrt 2}\right ).
\end{align*}

For $x>-\mu t$ we have  $\Prob{S^{(-\mu)}_t\leq x}=\Prob{S_t-\mu t\leq x} = 2(1-\Phi(t/\sqrt{x+\mu t})),$ thus the positivity function is given by
\begin{align}\label{eq:subordinator_positivity_function}
    p^{(-\mu)}_t  \coloneqq \Prob{S^{(-\mu)}_t>0} = 2\Phi\left ( \sqrt\frac{t}{\mu}\right )-1  = \erf\left (\sqrt\frac{t}{\mu} \right ).
\end{align}

Next, we apply our characterization of the occupation time for L\'evy processes, Proposition~\ref{prop:ot_levy_laplace-transform}, to compute the double Laplace transform of $A^{(-\mu)},$
  \begin{align}\label{eq:double_lt_stable}
    G(q, \lambda)=\Ex\left [\int_0^\infty e^{-qt}\exp(-\lambda A_t^{(-\mu)})dt\right ] &= q^{-1}\exp\left (-\int_0^\infty e^{-qt}\frac{p_t^{(-\mu)}}{t} (1-e^{-\lambda t}) dt\right ),
  \end{align}
for $q, \lambda>0,$  more explicitely. Notice that for $q, \lambda$ such that $\Re(q+\lambda)>0$ the integral $\int_0^\infty dt e^{-qt}(1-e^{-\lambda t})p_t/t$ exists, since $\abs{e^{-(q+\lambda)t}p_t/t}\leq e^{-\Re (q+\lambda)t}$ is integrable iff $\Re (q+\lambda)>0$. By our assumption $q, \lambda>0$ this condition is met. Writing $b\coloneqq 1/(4\mu),$ 
\begin{align*}
  & -\frac{d}{dq} \int_0^\infty \frac{p_t^{(-\mu)}}{t}e^{-qt}(1-e^{-\lambda t})dt\\
    &= \int_0^\infty p_t^{(-\mu)}e^{-qt}(1-e^{-\lambda t})dt\\
    &= \int_0^\infty p_t^{(-\mu)}e^{-qt}dt-\int_0^\infty p_t^{(-\mu)}e^{-(q+\lambda)t}dt\\
    &= \int_0^\infty e^{-qt}\erf\sqrt{bt}dt-\int_0^\infty e^{-(q+\lambda)t}\erf\sqrt{bt} dt\\
    &= \frac 1 q\sqrt\frac{b}{q+b}-\frac{1}{q+\lambda}\sqrt\frac{b}{q+\lambda+b},
\end{align*}
where in the penultimate step we apply $\int_0^\infty e^{-at}\erf\sqrt{bt}dt=\sqrt{b/(a+b)}/a$ for $a, b>0$, cf.~\cite[Equation (7.14.3)]{OlverLozierBoisvertClark2010}.
Notice that for any $\lambda>0$ and $l\geq 0$ a primitive function of $q\mapsto -\frac{1}{q+l}\sqrt\frac{b}{q+l+b}$ defined on $(0, \infty)$ is given by
\begin{align*}
  q\mapsto \log \left ( \frac{(\sqrt b+\sqrt{q+l+b})^2}{q+l} \right ),
\end{align*}
since
\begin{align*}
  \frac{d}{dq}\log\frac{(\sqrt b+\sqrt{q+l+b})^2}{q+\lambda} &= \frac{1}{\sqrt b+\sqrt{q+l+b}}\left (\frac{1}{\sqrt{q+l+b}}-\frac{\sqrt b+\sqrt{q+l+b}}{q+l}\right )\\
  &= \frac{\sqrt b-\sqrt{q+l+b}}{-(q+l)}\left ( \frac{\sqrt{q+l+b}}{q+l+b}-\frac{\sqrt b+\sqrt{q+l+b}}{q+l}   \right)\\
  &= \frac{(\sqrt b+\sqrt{q+l+b})(\sqrt b-\sqrt{q+l+b})}{(q+l)^2} - \frac{\sqrt{q+l+b}(\sqrt b-\sqrt{q+l+b})}{(q+\lambda)(q+\lambda+b)}\\
  &= \frac{-q-l}{(q+l)^2}-\frac{\sqrt b-\sqrt{q+l+b}}{(q+l)\sqrt{q+l+b}}\\
  &= -\frac{1}{q+l}\sqrt{\frac{b}{q+l+b}}.
\end{align*}

Applying this observation twice, namely for $l=0$ and $l=\lambda,$ we have
\begin{align*}
  -\int_0^\infty \frac{p_t^{(-\mu)}}{t}e^{-qt}(1-e^{-\lambda t})dt &= \log\left (  \frac{q}{(\sqrt b+\sqrt{q+b})^2} \cdot \frac{(\sqrt b+\sqrt{q+\lambda+b})^2}{q+\lambda} \right)+C,
\end{align*}
for some suitable real constant $C$. Since for any L\'evy process we have $G(q, \lambda)\to q^{-1}$  as $\lambda\downarrow 0$ from the definition of the double Laplace transform~\eqref{eq:ot_levy_laplace-transform}, we conclude $C=0$. Plugging into equation~\eqref{eq:double_lt_stable} yields
\begin{align}
  G(q, \lambda) &= \frac{1}{q+\lambda} \left ( \frac{\sqrt b+\sqrt{q+\lambda+b}}{\sqrt b+\sqrt{q+b}} \right )^2.
\end{align}

For the time being fix $\lambda, b$. For $u>0$ define the functions
\begin{align*}
	f(u) &\coloneqq  2e^{-\lambda u} \left (b(\erf\sqrt{bu}+1) +\sqrt{\frac{b}{\pi u}}e^{-bu}  \right ) ,\\
	g(u) &\coloneqq (2bu+1)(\erf\sqrt{bu}-1)+2e^{-bu}\sqrt{\frac{bu}{\pi}},
\end{align*}
Since $f, g$ are continuously differentiable, they are of bounded variation, cf.~\cite[Theorem 1.9 in Chapter 6]{Muresan2009}, and thus their respective Laplace transforms
\begin{align*}
	\bar f(q) &= \frac{(\sqrt b+\sqrt{q+\lambda+b})^2}{q+\lambda},\quad \bar g(q) =(\sqrt b+\sqrt{q+b})^{-2};\qquad q, \lambda>0,
\end{align*}
are uniquely determined Lebesgue almost everywhere, cf.~\cite[Theorem 6.3]{Widder1941}. Since $G=\bar f \bar g,$  $q\mapsto G(q, \lambda)$ is the Laplace transform of $f*g.$ Setting $h(u)\coloneqq 2\left (b(\erf\sqrt{bu}+1) +\sqrt{\frac{b}{\pi u}}e^{-bu}  \right )$ we find for the Laplace transform of $A_t^{(-\mu)}$:
\begin{align*}
  \Ex\left [\exp(-\lambda A_t^{(-\mu)})\right ] &=  (f*g)(t)=\int_0^t f(u)g(t-u)du = \int_0^t e^{-\lambda u}h(u)g(t-u)du.
\end{align*}
Consequently, the distribution of $A_t^{(-\mu)}$ has density
\begin{align*}
  x\mapsto h(x)g(t-x)\mathbf 1_{(0, t)}(x),
\end{align*}
and we conclude the claim.
\end{proof}

\bibliographystyle{amsplain}
\bibliography{literature}

\begin{affil}
{ }Mathematics Institute, University of Mannheim, 68131 Mannheim, Germany\\
Email: helmut dot pitters at gmail dot com
\end{affil}

\end{document}